\documentclass[sts]{imsart}
\input{commands.tex}
\newcommand{\EE}[1]{\underset{#1}{\mathop{\mathbb{E}}}}

\begin{document}
\begin{frontmatter}
\title{Entropic optimal transport is maximum-likelihood deconvolution}
\runtitle{Entropic OT is maximum-likelihood deconvolution}

\begin{aug}
\author{\fnms{Philippe}~\snm{Rigollet}\thanksref{t1}\ead[label=rigollet]{rigollet@math.mit.edu}}
\and
\author{\fnms{Jonathan}~\snm{Weed}\thanksref{t2}\ead[label=jon]{jweed@mit.edu}}

\affiliation{Massachusetts Institute of Technology}
\thankstext{t1}{This work was supported by NSF DMS-1712596, NSF TRIPODS-1740751, ONR N00014-17-1-2147, the Chan Zuckerberg Initiative DAF 2018-182642, and the MIT Skoltech Seed Fund.}
\thankstext{t2}{This work was supported in part by NSF Graduate Research Fellowship DGE-1122374.}

\address{{Philippe Rigollet}\\
{Department of Mathematics} \\
{Massachusetts Institute of Technology}\\
{77 Massachusetts Avenue,}\\
{Cambridge, MA 02139-4307, USA}\\
\printead{rigollet}
}

\address{{Jonathan Weed}\\
{Department of Mathematics} \\
{Massachusetts Institute of Technology}\\
{77 Massachusetts Avenue,}\\
{Cambridge, MA 02139-4307, USA}\\
\printead{jon}
}

\runauthor{Rigollet and Weed}
\end{aug}

\begin{abstract}
We give a statistical interpretation of entropic optimal transport by showing that performing maximum-likelihood estimation for Gaussian deconvolution corresponds to calculating a projection with respect to the entropic optimal transport distance.
This structural result gives theoretical support for the wide adoption of these tools in the machine learning community.
\end{abstract}

\end{frontmatter}

\section{Introduction}
Optimal transport is a fundamental notion with arising in several branches of mathematics, including probability, analysis and statistics. More recently, it has found new applications in computational domains such as machine learning and image processing~~\cite{ArjChiBot17,BonPanPar11,CouFlaTui17,RubTomGui00,SolDe-Pey15,JitSzaChw16,MueJaa15,RolCutPey16}. This newfound utility was largely fueled by algorithmic advances allowing optimal transport distances to be computed quickly between large scale discrete distributions~\cite{PeyCut17}.
At the heart of these algorithmic techniques is the idea of entropic penalization, which has been leveraged to obtain near-linear-time approximation schemes for optimal transport distances~\cite{Wil69a,Cut13,AltWeeRig17}. Hereafter, we refer to this technique as \emph{entropic optimal transport}. This line of well established computational research stands in sharp contrast with our \emph{statistical} understanding of regularization for optimal transport, which is still in its infancy
\cite{ForHutNit18}. 
Entropic regularization has been shown to play a central statistical role in a variety of problems related to model selection~\cite{JudRigTsy08,Rig12,RigTsy11,RigTsy12, DalTsy08,DalTsy12,DalTsy12b}  but our knowledge of its effect on optimal transport is currently limited to  experimental evidence~\cite{Cut13,PeyCut17} without theoretical support.

In this note, we give a statistical interpretation of entropic optimal transport, showing that under some modeling assumptions, it corresponds to the objective function in maximum-likelihood estimation for deconvolution problems involving additive Gaussian noise.
This interpretation provides a first indication that optimal transport problems where data is subject to Gaussian observation error should be handled with entropic regularization. Moreover, our results indicate that in the same context, a relaxed version of optimal transport should be preferred, as it is equivalent to maximum-likelihood estimation even in absence of said modeling assumptions.

\section{Entropic optimal transport}

Throughout we denote by $\gamma$ a probability measure on $\R^d \times \R^d$ and by $\|\cdot\|$ the Euclidean norm over $\R^d$.
Given such a measure, we denote by $\pi_X \gamma$ and $\pi_Y \gamma$ the two measures on $\R^d$ obtained by projecting onto the first and second component, respectively.
Given probability measures $\mu$ and $\nu$ on $\R^d$, define
\begin{equation*}
\cM(\mu, \nu) \defeq \{\gamma : \pi_X \gamma = \mu, \pi_Y \gamma = \nu\}\quad \text{and} \quad
\cM(\mu) \defeq \{\gamma: \pi_Y \gamma = \mu\}\,.
\end{equation*}
We also recall the definition of Kullback-Leibler (KL) divergence between probability measures $\mu$ and $\nu$:
$$
D(\mu \| \, \nu)=\left\{
\begin{array}{ll}
\int \log \big(\frac{\dd \mu}{\dd \nu}\big) \dd \mu,& \text{if } \mu\ll \nu\\
\infty, & \text{otherwise.}
\end{array}\right.
$$

\begin{defin}
The \emph{entropic optimal transport distance} between $\mu$ and $\nu$ is
\begin{equation}
\Went(\mu, \nu) \defeq \min_{\gamma \in \cM(\mu, \nu)} \int \frac 1 2 \|x - y\|^2 \dd\gamma(x, y) + \sigma^2 I(\gamma)\,,\label{eq:ent_def}
\end{equation}
where $I(\gamma)$ is the \emph{mutual information} defined by
\begin{equation*}
I(\gamma) \defeq D(\gamma \, \| \, \pi_X \gamma \otimes \pi_Y \gamma)\,.
\end{equation*}
\end{defin}
Note that when $\sigma=0$, this corresponds the squared 2-Wasserstein distance between probability measures over $\R^d$~\cite{San15}. When $\sigma > 0$, $\Went$ no longer satisfies the axioms of a (squared) distance, but it still possesses useful distance-like properties~\cite{Cut13}.
We employ the term ``distance'' for all values of $\sigma$ for terminological consistency.
When $\mu$ and $\nu$ are discrete measures, the minimizer of~\eqref{eq:ent_def} is also discrete and agrees with the minimizer of
\begin{equation}
\label{eqn:entropicOT}
\min_{\gamma \in \cM(\mu, \nu)} \int \frac 1 2 \|x - y\|^2 \dd\gamma(x, y) - \sigma^2 H(\gamma)\,,
\end{equation}
where $H$ is the standard Shannon entropy:
$$
H(\gamma) \defeq \sum_{ij} \gamma_{ij} \log \frac{1}{\gamma_{ij}}\,,
$$
where $\gamma_{ij} \defeq \gamma(x_i, y_j)$ for $(x_i, y_j) \in \mathrm{supp}(\gamma)$.
Note that definition~\eqref{eqn:entropicOT} is the one proposed by~\cite{Cut13} for discrete measures, whereas definition~\eqref{eq:ent_def} corresponds to the appropriate generalization studied in~\cite{GenCutPey16} for measures that are not necessarily discrete.
It is not hard to check that $\Went(\mu, \nu)<\infty$ for any pair $(\mu, \nu)$ possessing finite second moments, since the independent coupling $\mu \otimes \nu$ in the minimization that appears in~\eqref{eq:ent_def} leads to a finite objective value.

A maximum-likelihood interpretation of entropic optimal transport is already known in the context of a large-deviation principle for Brownian motion~\cite{Leo14}. In this context, given two distributions $\mu$ and $\nu$ (viewed as the positions of particles at times $t = 0$ and $t = 1$), Schr\"odinger~\cite{Sch32} gave a heuristic argument motivated by statistical physics establishing that the law of independent particles undergoing Brownian motion conditioned on having initial and final distributions $\mu$ and $\nu$ respectively is induced by the solution to the optimal transport problem with entropic regularization~\cite{Leo14}.
 While suggestive, this interpretation does not hold any immediate implications for estimation problems where only data is available rather than distributions $\mu$ and $\nu$. 
Below, we introduce the classical deconvolution model for corrupted observations and show that in this context, entropic optimal transport is precisely the maximum-likelihood estimator.

\section{Deconvolution}

Let $\cP$ be a given family of probability distributions over $\R^d$ with finite second moments and let $P^*$ be an unknown distribution, also with finite second moment. The deconvolution problems consists in estimating $P^*$ on the basis of corrupted observations $Y_1, \ldots, Y_n$, where 
\begin{equation}
\label{EQ:gaussian_deconv}
Y_i = X_i + Z_i\,, \qquad i=1, \ldots, n\,,
\end{equation}
and the errors $Z_1, \ldots, Z_n$ are independent of $X_1, \ldots, X_n$. 
For identifiability purposes, the random variables $\{Z_i\}$ are assumed to be independent copies of a random variable $Z$ with known distribution: $Z \sim \cN(0,\sigma^2)$ where the variance $\sigma^2$ is known.

In this context, the distribution of $Y_i$ admits a density $\varphi_\sigma \star \dd P^*$ with respect to $\lambda$ where, for any $P \in \cP$, we define
\begin{equation}
\label{EQ:def_conv}
\varphi_\sigma \star \dd P(y)=\int \varphi_{\sigma}(y-x) \dd P(x)
\end{equation}
and $\varphi_\sigma$ denotes the density of $Z\sim \cN(0,\sigma^2)$. Under these assumptions, we call~\eqref{EQ:gaussian_deconv} the \emph{Gaussian deconvolution model}.

Deconvolution is a classical question of nonparametric statistics~\cite{CarHal88,Fan91,CaiChaDed11} and is core to mixture models~\cite{lin95} as well as statistical models with measurement errors~\cite{CarRupSte06}. As such, it has received significant attention from the statistics literature. More recently, it was shown that deconvolution has strong methodological and mathematical connections to optimal transport in the context of a problem known as uncoupled regression~\cite{RigWee18}.

A natural candidate to estimate $P^*$ is the maximum-likelihood estimator (MLE) $\hat P$ defined by
\begin{equation}
\label{EQ:def_MLE}
\hat P = \argmax_{P \in \cP} \sum_{i=1}^n \log\varphi_\sigma \star \dd P(Y_i)\,,
\end{equation}
The statistical properties of the MLE are well known and have been established under general conditions on the class $\cP$~\cite{BicDok06, GinNic16}. In the next section, we show that entropic optimal transport is in fact implementing~$\hat P$.  

\section{Entropic optimal transport is maximum-likelihood deconvolution}

In this section, we adopt the Gaussian deconvolution model~\eqref{EQ:gaussian_deconv} of the previous section. The extension to other distributions for the corruption errors $\{Z_i\}$ is postponed to Section~\ref{sec:extensions}.

Our main result involves families of distributions satisfying a particular closure condition.%
\begin{defin}
A class $\cP$ of probability measures is said to be \emph{closed under domination} if $Q \ll P$ for some $P \in \cP$ implies that  $Q \in \cP$.
\end{defin}
Many families are closed under domination. For example  the class of all measures, the class of all measures absolutely continuous with respect to some reference measure $\sigma$, the class of discrete measures, and the set of measures whose support is finite or contains at most $k$ points all possess this property. A class $\cP$ of probability measures may always be augmented to be closed under domination by adding to it the set of probability measures $\bigcup_{P \in \cP}\{Q\,:\, Q \ll P\}$. The extension to families not closed under domination is considered in Section~\ref{sec:extensions}.

We are now in a position to state our main result: a structural representation of the maximum-likelihood estimator $\hat P$ in terms of entropic optimal transport.
\begin{thm}\label{thm:main}
Let $\cP$ be a class of probability measures that is closed under domination and assume the Gaussian deconvolution model~\eqref{EQ:gaussian_deconv}.
Then the maximum-likelihood estimator $\hat P$ over $\cP$ defined in~\eqref{EQ:def_MLE} satisfies:
\begin{equation*}
\hat P = \argmin_{P \in \cP} \Went\Big(P, \frac 1 n \sum_{i=1}^n \delta_{y_i}\Big)\,.
\end{equation*}
In other words, the maximum-likelihood estimator $\hat P$ is the projection of the empirical measure $\frac 1n \sum_{i=1}^n \delta_{y_i}$ onto $\cP$ with respect to the entropic optimal transport distance $\Went$.
\end{thm}

The projection estimator $\argmin_{P \in \cP} \Went\Big(P, \frac 1 n \sum_{i=1}^n \delta_{y_i}\Big)$ has been employed in the machine learning community~\cite{MonMulCut16,GenPeyCut18} as a smoothed version of a minimum Kantorovich distance estimator~~\cite{BasBodReg06} more suitable for optimization. Theorem~\ref{thm:main} shows that this estimator has a statistical interpretation in addition to its computational benefits.

As noted above, in the special case when $\sigma^2 = 0$, the quantity $\Went$ reduces to the squared $2$-Wasserstein distance $W$. In the context of Gaussian mixture models, when $\cP$ is the class of probability distributions supported on at most $k$ points, solving $\argmin_{P \in \cP} W\Big(P, \frac 1 n \sum_{i=1}^n \delta_{y_i}\Big)$ corresponds to performing a ``hard'' clustering of the data by minimizing the $k$-means objective. It is known, however, that hard $k$-means clustering does not lead to consistent estimation of the centroids in a mixture of Gaussians model, whereas consistent estimation \emph{can} be achieved with the MLE, which induces a relaxed ``soft'' clustering~\cite{KeaManNg97}.
Theorem~\ref{thm:main} implies that replacing $W$ by $\Went$ precisely corresponds to this relaxation.

\begin{proof}
Write for simplicity $\ell_P \defeq \log\varphi_\sigma \star \dd P$.
By~\eqref{EQ:def_conv}, we have
\begin{equation*}
\ell_P(y_i) = C + \log \int \exp\big(- \frac{1}{2\sigma^2} \|x - y_i\|^2\big) \dd P(x)\,,
\end{equation*}
where $C$ is a constant not depending on $y_i$ or $P$.
The Gibbs variational principle~\cite[Equation~(5.2.1)]{Cat04} then implies that
\begin{equation*}
\ell_P(y_i) = C - \frac{1}{\sigma^2}\min_{Q_i} \left\{ \frac 12 \EE{x \sim Q_i}\|x - y_i\|^2 + \sigma^2D(Q_i \, \| \, P) \right\}\,,
\end{equation*}
where the minimization is taken over probability measures on $\R^d$.
Then, by definition of the MLE, we have
\begin{equation*}
\hat P = \argmin_{P \in \cP} \min_{Q_1, \dots, Q_n} \frac 1 n \sum_{i=1}^n \left[\frac 1 2\EE{x \sim Q_i}  \|x - y_i\|^2 + \sigma^2D(Q_i \, \| \, P) \right]\,.
\end{equation*}
Next, given any set of $n$ distributions $\{Q_1, \dots, Q_n\}$ on $\R^d$, we can define the joint probability measure $\bar \gamma$
on $\R^d \times \{y_1, \dots, y_n\}$ by
\begin{equation*}
\bar \gamma \defeq \frac 1n \sum_{i=1}^n Q_i \otimes \delta_{y_i}\,.
\end{equation*}
Note that $\pi_Y \gamma = U$, the uniform distribution on $\{y_1, \dots, y_n\}$.
Conversely, for any joint probability measure $\bar \gamma$ on $\R^d \times \R^d$ satisfying $\pi_Y \gamma = U$, we can decompose $\gamma$ as $\frac 1n \sum_{i=1}^n Q_i \otimes \delta_{y_i}$ for some $Q_1, \dots, Q_n$.
This bijection between sets of $n$ probability measures $\{Q_1, \dots, Q_n\}$ on $\R^d$ and joint measures $\bar \gamma \in \cM(U)$ satisfies the equality
\begin{equation*}
\frac 1 n \sum_{i=1}^n \left[ \frac 1 2\EE{x \sim Q_i} \|x - y_i\|^2 + \sigma^2D(Q_i \, \| \, P) \right] = \frac 1 2\EE{(x,y)\sim\bar \gamma}  \|x - y\|^2 + \sigma^2D(\bar \gamma \, \| \, P \otimes U)\,.
\end{equation*}

We can therefore leverage this bijection to rewrite $\hat P$ as
\begin{equation*}
\hat P = \argmin_{P \in \cP} \min_{\gamma \in \cM(U)}\left\{\frac 1 2  \EE{(x,y)\sim \gamma} \|x - y\|^2 + \sigma^2D(\gamma \, \| \, P \otimes U)\right\}\,.
\end{equation*}

By Lemma~\ref{lem:product_decomp},
\begin{align*}
D(\gamma \, \| \, P \otimes U)
 = I(\gamma) + D(\pi_X \gamma \| P)\,,
\end{align*}
where we have used that  $D(\pi_Y \gamma \| U) = 0$ for any $\gamma \in \cM(U)$. We obtain
$$
\hat P = \argmin_{P \in \cP}V(P)\,,
$$
where 
\begin{equation}
\label{EQ:defV}
V(P):= \min_{\gamma \in \cM(U)}\left\{\frac 1 2  \EE{(x,y)\sim \gamma} \|x - y\|^2 + \sigma^2I(\gamma) + \sigma^2D(\pi_X \gamma \| P)\right\}
\end{equation}

Next, for any $P \in \cP$, denote by $\gamma_P$ the coupling that achieves the minimum in the definition~\eqref{eq:ent_def} of $\Went(P, U)$ and observe that since $D(\pi_X \gamma_P \| P)=0$, we get
\begin{align*}
V(P)
\le \frac 1 2  \EE{(x,y)\sim \gamma_P} \|x - y\|^2 + \sigma^2I(\gamma_P)=W_{\sigma^2}(P,U)<\infty\,,
\end{align*}
since $P$ has finite second moment.
We now show that the functions $P \mapsto  V(P)$ and $P \mapsto W_{\sigma^2}(P,U)$
achieve their minimum over $\cP$ at the same $P$.
To that end, observe that since $V(P)<\infty$, the minimum in the definition~\eqref{EQ:defV} of $V$ may be restricted to couplings $\gamma$ such that $\pi_X\gamma \ll P$, since otherwise $D(\pi_X \gamma \| P)$ is infinite. 
Thus, 
\begin{align}
\min_{P \in \cP}V(P)&=\min_{P \in \cP}\min_{\substack{\gamma \in \cM(U)\\\pi_X \gamma \ll P}}\left\{\frac 1 2  \EE{(x,y)\sim \gamma} \|x - y\|^2 + \sigma^2I(\gamma) + \sigma^2D(\pi_X \gamma \| P)\right\}\label{EQ:unbalanced}\\
& \ge \min_{P \in \cP}\min_{\substack{\gamma \in \cM(U)\\\pi_X \gamma \ll P}}\left\{\frac 1 2  \EE{(x,y)\sim \gamma} \|x - y\|^2 + \sigma^2I(\gamma)\right\}\nonumber\\
&= \min_{P \in \cP}\min_{\gamma \in \cM(P,U)}\left\{\frac 1 2  \EE{(x,y)\sim \gamma} \|x - y\|^2 + \sigma^2I(\gamma) \right\}\nonumber\\
&=\min_{P \in \cP}W_{\sigma^2}(P,U)\,,\nonumber
\end{align}
where in the inequality we used that $D(\pi_X \gamma \| P) \geq 0$ and in the second equality we used that $\cP$ is closed under domination.
In light of the previous two displays, we conclude that $\hat P = \argmin_{P \in \cP} \Went\Big(P, \frac 1n \sum_{i=1}^n \delta_{y_i}\Big)$, as claimed.

\end{proof}

\section{Extensions}
\label{sec:extensions}

\subsection{Relaxed transport}
Traditional parametric classes of distributions $\cP$ are often not closed under domination. For example, the one-dimensional scale/location family with template density $\varphi$ with respect to the Lebesgue measure $\mathrm{Leb}$ on $\R$ is defined by
$$
\cP=\Big\{P\,:\, \frac{\dd P}{\dd \mathrm{Leb}}(\cdot)=\frac{1}{\tau}\varphi\big(\frac{\cdot - \mu}{\tau}\big)\,, \mu \in \R, \tau>0\Big\}\,.
$$
Clearly $\cP$ is not closed under domination. In such cases, Theorem~\ref{thm:main} can fail to hold as illustrated by Proposition~\ref{PROP:fail} below. However, it follows from~\eqref{EQ:unbalanced} in the proof of Theorem~\ref{thm:main} that the following representation for the MLE \emph{always} holds:
$$
\hat P=\argmin_{P \in \cP} W_{\sigma^2}^{\textsf{rel}}\big(P, \frac{1}{n} \sum_{i=1}^n \delta_{y_i}\big)\,,
$$
where $W_{\sigma^2}^{\textsf{rel}}$ denotes the \emph{relaxed entropic optimal transport} distance defined for any probability measures $\mu, \nu$ by
$$
 W_{\sigma^2}^{\textsf{rel}}(\mu, \nu)=\min_{\substack{\gamma \in \cM(\nu) \\ \pi_X \gamma \ll \mu}}  \int \frac 1 2 \|x - y\|^2 \dd \gamma(x, y) + \sigma^2\big[ I(\gamma) + D(\pi_X \gamma \, \| \, \mu)\big]\,.
$$
This result indicates that it may be preferable to use relaxed transport in statistical contexts.
Relaxing the marginal constraints in the optimal transport problem is an idea which has attracted significant recent interest ~\cite{FroZhaMob15,LieMieSav18,ChiPeySch16} after it was first formally proposed in~\cite{Ben03} under the name ``unbalanced transport'' to generalize optimal transport to apply to nonnegative measures with different total mass. Relaxed optimal transport has since been used to improve robustness to sampling noise in statistical applications~\cite{SchShuTab17}.

We now exhibit a simple example of a class $\cP$ that is not closed under domination and for which Theorem~\ref{thm:main} fails to hold. For any $\sigma > 0$, let $\cP=\{P_1, P_2\}$ where $P_1$ and $P_2$ are two probability measures on the real line defined respectively by 
$$
P_1 \defeq \frac{1}{2} (\delta_{0} + \delta_{4\sigma})\quad \text{and} \quad P_2 \defeq \frac 12 (\delta_{2\sigma} + \delta_{6 \sigma})
$$

Let $X\sim P_1$ and let $Y = X+Z$ where $Z \sim \cN(0, \sigma^2)$ is independent of $X$. 
\begin{prop}
\label{PROP:fail}
With probability at least $.15$, we have
$$
P_1  = \argmin_{P \in \cP} \Went(P, \delta_{Y})\,, \quad \text{and} \quad 
P_2  = \argmax_{P \in \cP} \log \varphi_\sigma \star \dd P(Y)\,.
$$
In other words, the maximum-likelihood estimator and the projection with respect to the entropic optimal transport distance do not agree.
\end{prop}
\begin{proof}
By rescaling, it suffices to consider the case $\sigma^2 = 1$.
For each $P \in \cP$, the set $\cM(P, \delta_Y)$ contains only the independent coupling $P \otimes \delta_Y$, for which the mutual information vanishes.
Therefore,
\begin{equation*}
\Went(P_1, \delta_Y) = W_0(P_1, \delta_Y) = \frac 14 (Y^2 + (Y - 4)^2) \quad \quad \Went(P_2, \delta_Y) = \frac 14 ((Y-2)^2 + (Y - 6)^2)\,,
\end{equation*}
so that $P_1$ is the unique minimizer over $\cP$ of $\Went(P, \delta_Y)$ as long as $Y < 3$.

On the other hand, $P_2$ is the unique maximum-likelihood estimator if
\begin{equation*}
e^{-Y^2/2} + e^{-(Y-4)^2/2} < e^{-(Y -2)^2/2} + e^{-(Y - 6)^2/2}\,,
\end{equation*}
and it can be checked that this condition holds on the interval $[1.01,3)$.

Therefore, if $Y \in [1.01,3)$, then the claimed situation occurs. To conclude the proof, it suffices to observe that 
$P_1(1.01 \le Y <3) \ge .5 \PP(1.01\le |Z| \le 2.99) \ge .15$.
\end{proof}

\subsection{General noise distribution}

In the is section, we raise the question of non-Gaussian deconvolution that arises when the errors $\{Z_i\}$ are not Gaussian. It turns out that a simple modification of our argument can be made to accommodate any noise distribution that admits a density $f$ with respect to the Lebesgue measure on $\R^d$. In this context, the MLE takes the form
\begin{equation}
\label{EQ:MLE_gen_noise}
\hat P = \argmax_{P \in \cP} \sum_{i=1}^n \log f \star \dd P(Y_i)
\end{equation}

The use of the squared Euclidean norm $\frac{1}{2\sigma^2}\|x-y\|^2$ in the objective~\eqref{eq:ent_def} is tailored to Gaussian errors and may be replaced with $-\log f(x-y)$. After rescaling by $\sigma^2$, we define
$$
W_f(\mu, \nu) \defeq \min_{\gamma \in \cM(\mu, \nu)} -\int \log f(x - y) \dd\gamma(x, y) + I(\gamma)\,.
$$
We assume in what follows that $\log f(\cdot - y)  \in L_1(P)$ for all $y \in \R^d$ and $P \in \cP$.
The following proposition is stated without proof as it follows from exactly the same arguments as Theorem~\ref{thm:main}.
\begin{prop}
\label{PROP:gen_noise}
Let $\cP$ be a class of probability measures that is closed under domination and assume the deconvolution model~\eqref{EQ:gaussian_deconv} where $Z_i$ has density $f$ with respect to the Lebesgue measure.
Then the maximum-likelihood estimator $\hat P$ over $\cP$ defined in~\eqref{EQ:MLE_gen_noise} satisfies:
\begin{equation*}
\hat P = \argmin_{P \in \cP} W_f\Big(P, \frac 1 n \sum_{i=1}^n \delta_{y_i}\Big)\,.
\end{equation*}
In other words, the maximum-likelihood estimator $\hat P$ is the projection of the empirical measure $\frac 1n \sum_{i=1}^n \delta_{y_i}$ onto $\cP$ with respect to the entropic optimal transport distance $W_f$.
\end{prop}
In the case where $f(z) \propto \exp(-\|z\|_p^p)$, the cost $- \log f(x - y)$ corresponds to the $\ell_p$ metric arising in the definition of the $p$-Wasserstein distance. Another intriguing example is the cost $-\log \cos^2(\|x-y\|\wedge \pi/2)$, which appears in the definition of the Wasserstein-Fisher-Rao~\cite{ChiPeySch16} or Hellinger-Kantorovich~\cite{LieMieSav18} distance between positive measures. These formulations differ from ours in that they consider a version of relaxed transport which allows for misspecification of both marginals; nevertheless, our analysis suggests that inference involving these distances is likely to be robust to convolutional noise $Z$ supported on the Euclidean ball of radius $\pi/2$ around the origin with density
$$
f(z) \propto \cos^2(\|z\|)\1\big(\|z\| \le \frac{\pi}2\big)\,.
$$
The Wasserstein-Fisher-Rao/Hellinger-Kantorovich distance is motivated by a dynamic formulation of unbalanced transport and it is unclear whether the above noise distribution plays a special role in the context of deconvolution.
\section{Additional lemmas}

\begin{lemma}\label{lem:product_decomp}
Let $\gamma$ be a measure on $\R^d \times \R^d$, and let $\alpha$ and $\beta$ be probability measures on $\R^d$.
Then
\begin{equation*}
D(\gamma \, \| \, \alpha \otimes \beta) = I(\gamma) + D(\pi_X \gamma \, \| \, \alpha) + D(\pi_Y \gamma \, \| \, \beta)\,.
\end{equation*}
\end{lemma}
\begin{proof}
We assume $\gamma \ll \pi_X \gamma \otimes \pi_Y \gamma \ll \alpha \otimes \beta$, since otherwise both sides are infinite.
Under this condition, we have
\begin{align*}
D(\gamma \, \| \, \alpha \otimes \beta) & = \int \log \frac{\mathrm{d}\gamma}{\mathrm{d}\alpha\mathrm{d}\beta}(x, y) \dd \gamma(x, y) \\
& = \int \log \frac{\mathrm{d}\gamma}{\mathrm{d}\pi_X \gamma \mathrm{d}\pi_Y \gamma}(x, y) \dd \gamma(x, y) + \int \log \frac{\mathrm{d}\pi_X \gamma \mathrm{d}\pi_Y \gamma}{\mathrm{d}\alpha\mathrm{d}\beta}(x, y) \dd \gamma(x, y) \\
& = I(\gamma) + \int \log \frac{\mathrm{d}\pi_X \gamma}{\mathrm{d}\alpha}(x, y) \dd \gamma(x, y) + \int \log \frac{\mathrm{d}\pi_Y \gamma}{\mathrm{d}\beta}(x, y) \dd \gamma(x, y) \\
& = I(\gamma) + \int \log \frac{\mathrm{d}\pi_X \gamma}{\mathrm{d}\alpha}(x) \dd \pi_X\gamma(x) + \int \log \frac{\mathrm{d}\pi_Y \gamma}{\mathrm{d}\beta}(y) \dd \pi_Y\gamma(y) \\
& = I(\gamma) + D(\pi_X \gamma \, \| \, \alpha) + D(\pi_Y \gamma \, \| \, \beta)\,.
\end{align*}
\end{proof}

\bibliographystyle{alphaabbr}
\bibliography{weedmain,MLEOT}
\end{document}